\RequirePackage{fix-cm}
\documentclass[smallextended]{svjour3}       
\smartqed  
\usepackage{graphicx}
\usepackage{mathptmx}      
\usepackage{latexsym}
\usepackage{amsmath}
\usepackage{lipsum}
\usepackage{amsfonts}
\usepackage{graphicx}
\usepackage{epstopdf}
\usepackage{algorithmic}
\usepackage{algorithm}
\usepackage{bbm}
\usepackage{amsopn}
\usepackage{placeins}
\newcommand{\lek}{\left[}
\newcommand{\rek}{\right]}

\newcommand{\prob}{\stackrel{\PP}{\longrightarrow}}

\newcommand{\PP}{\mathbb{P}}
\newcommand{\E}{\mathbb{E}}

\spnewtheorem{assumptions}{Assumption}{\bf}{\it}
 \def\makeheadbox{\relax}
\begin{document}

\title{Beyond the Bakushinkii veto: Regularising linear inverse problems without knowing the noise distribution
}

\titlerunning{Regularising linear inverse problems without knowing the noise distribution}        

\author{Bastian Harrach \and
        Tim Jahn        \and
        Roland Potthast 
}


\institute{Bastian Harrach  \at
              Institute for Mathematics\\
              Goethe-Universit\"at Frankfurt \\
              Frankfurt am Main\\
              \email{harrach@math.uni-frankfurt.de}\\              
              \and
              Tim Jahn \at
              Institute for Mathematics\\
              Goethe-Universit\"at Frankfurt \\
              Frankfurt am Main\\
              Tel.: +49-69-79822715\\
              \email{jahn@math.uni-frankfurt.de}\\
                \and
              Roland Potthast  \at
              Data Assimilation Unit\\
              Deutscher Wetterdienst\\
              Offenbach am Main\\
              \email{Roland.Potthast@dwd.de}\\         
}

\date{Received: date / Accepted: date}

\maketitle

\begin{abstract}
This article deals with the solution of linear ill-posed equations in Hilbert spaces.
Often, one only has a corrupted measurement of the right hand side at hand and the Bakushinskii veto tells us, that we are not able to solve
the equation if we do not know the noise level. But in applications it is ad hoc unrealistic to know the error of a measurement. In practice, the error of a measurement may often be estimated through averaging of multiple measurements. We integrated that in our anlaysis and obtained convergence to the true solution, with the only assumption that the measurements are unbiased, independent and identically distributed according to an unknown distribution.
\end{abstract}

\section{Introduction}
\label{intro}
The goal is to solve the ill-posed equation $K\hat{x}=\hat{y}$, where $\hat{x}\in\mathcal{X}$ and $\hat{y}\in\mathcal{Y}$ are elements of infinite dimensional Hilbert spaces and $K$ is either linear and bounded with non-closed range, or more specifically  compact.
 We do not know the right hand side $\hat{y}$ exactly, but we are given several measurements $Y_1,Y_2,...$ of it, which
are independent, identically distributed and unbiased ($\E Y_i = \hat{y}$) random variables. Thus we assume, that we are able to measure the right hand side multiple times, and a crucial requirement is that the solution does not
change at least on small time scales. Let us stress that using multiple measurements to decrease the data error is a standard engineering practice under the name ’signal averaging’, see, e.g., \cite{lyons2004understanding} for an introducing monograph or \cite{hassan2010reducing} for a survey article. Examples with low or moderate numbers of measurements (up to a hundred) can be found in \cite{buades2009note} or \cite{mackay2004high} on image averaging or \cite{garcia2014retracking} on satellite radar measurements.
For the recent first image of a black hole, even up to $10^9$ samples were averaged, cf. \cite{akiyama2019first}.\\
The given multiple measurements naturally lead to an estimator of $\hat{y}$, namely the sample mean

\begin{equation*}
 \bar{Y}_n:=\frac{\sum_{i \le n}Y_i}{n}.
\end{equation*}

\noindent But, in general $K^+\bar{Y}_n \not\to K^+\hat{y}$ for $n\to\infty$, because the generalised inverse (Definition 2.2 of \cite{engl1996regularization}) of $K$ is not continuous.
So the inverse is replaced with a family of continuous approximations $(R_{\alpha})_{\alpha>0}$, called regularisation, e.g. the Tikhonov 
regularisation $R_{\alpha}:=\left(K^*K+\alpha Id\right)^{-1}K^*$, where $Id:\mathcal{X}\to\mathcal{X}$ is the identity. 
The regularisation parameter $\alpha$ has to be chosen accordingly to the data $\bar{Y}_n$ and the true data error 

\begin{equation*}
 \delta_n^{true}:=\| \bar{Y}_n-\hat{y}\|,
\end{equation*}

\noindent which is also a random variable. Since $\hat{y}$ is unknown, $\delta_n^{true}$ is also unkown and has to be guessed. Natural guesses are

\begin{equation*}
\delta_n^{est}:= \frac{1}{\sqrt{n}}\quad \mbox{or} \quad \delta_n^{est}:=\frac{\sqrt{\sum_{i\le n}\| Y_i -\bar{Y}_n\|^2/(n-1)}}{\sqrt{n}}.
\end{equation*}

\noindent  One first natural approach is now to use a (deterministic) regularisation method together with $\bar{Y}_n$ and $\delta_n^{est}$.
We are in particular interested in the discrepancy principle \cite{morozov1968error}, wich is known to provide optimal convergence rates (for some $\hat{y}$) in the classical deterministic setting. The following main result states, that in a certain sense, the natural approach converges and yields the optimal deterministic rates asymptotically. 

\begin{corollary}[to Theorem \ref{thwos} and \ref{thws}]\label{cor1}
 Assume that $K:\mathcal{X}\to \mathcal{Y}$ is a compact operator with dense range between Hilbert spaces and that $Y_1,Y_2,...$ are i.i.d. $\mathcal{Y}-$valued random variables which fullfill $\E[ Y_1] = \hat{y}\in \mathcal{R}(K)$ and $0<\E\lVert Y_1-\hat{y}\rVert^2<\infty$. Define the Tikhonov regularisation $R_{\alpha}:=\left(K^*K+\alpha Id\right)^{-1}K^*$ (or the truncated singular value regularisation, or Landweber iteration). Determine $(\alpha_n)_n$ through the discrepancy principle using $\delta_n^{est}$ (see Algorithm $1$). Then $R_{\alpha_n}\bar{Y}_n$ converges to $K^+\hat{y}$ in probability, that is

\begin{equation*}
\mathbb{P}\left(\lVert R_{\alpha_n}\bar{Y}_n-K^+\hat{y}\rVert \le \varepsilon\right)\to 1,\quad n\to\infty,\quad\forall\varepsilon>0.
\end{equation*}

\noindent Moreover, if $K^+\hat{y}=\left(K^*K\right)^{\nu/2}w$ with $w\in\mathcal{X}$ and $\lVert w\rVert\le\rho$ for $\rho>0$ and $0<\nu<\nu_0-1$ (where $\nu_0$ is the qualification of the chosen method, see Assumptions \ref{assumptions1}), then for all $\varepsilon>0$,

\begin{equation*}
\mathbb{P}\left(\lVert R_{\alpha_n}\bar{Y}_n-K^+\hat{y}\rVert\le \rho^\frac{1}{\nu+1}\left( \frac{1}{\sqrt{n}}\right)^{\frac{\nu}{\nu+1}-\varepsilon}\right)\to 1,\quad n\to\infty.
\end{equation*}

\end{corollary}

\noindent Moreover it is shown, that the approach in general does not yield $L^2$ convergence \footnote{also called convergence of the integrated mean squared error or root mean squared error}
for a naive use of the discrepancy principle, but it does for a priori regularisation. We also discuss quickly, how one has to estimate the error to obtain almost sure convergence.\\
\noindent To solve an inverse problem, as already mentioned, typically some a priori information about the noise is required. This may be, in the classical deterministic case,
the knowledge of an upper bound of the noise level, or, in the stochastic case, some knowledge of the error distribution or the restriction
to certain classes of distributions, for example to Gaussian distributions.
Here we present the first rigorous convergence theory for noisy measurements without any knowledge of the error distribution. The approach can be easily used by everyone, who can measure multiple times.\\
Stochastic or statistical inverse problems are an active field of research with close ties to high dimensional statistics
(\cite{gine2016mathematical},\cite{ghosal2017fundamentals},\cite{nakamura2015inverse}).
In general, there are two approaches to tackle an ill-posed problem with stochastic noise. The Bayesian setting considers the solution of the problem itself as
a random quantity, on which one has some a priori knowledge (see \cite{kaipio2006statistical}). This opposes the frequentist setting, where the inverse problem is
assumed to have a deterministic, exact solution (\cite{cavalier2011inverse},\cite{bissantz2007convergence}). We are working in the frequentist setting, but we stay 
close to the classic deterministic theory of linear inverse problems (\cite{engl1996regularization},\cite{rieder2013keine},\cite{tikhonov1977methods}). 
 For statistical inverse problems, typical methods to determine the regularisation parameter are cross validation \cite{wahba1977practical}, Lepski's balancing principle 
 \cite{mathe2003geometry} or penalised empirical risk minimisation \cite{cavalier2006risk}. Modifications of the discrepancy principle were studied recently
 (\cite{blanchard2012discrepancy},\cite{lu2014discrepancy},\cite{blanchard2018optimal},\cite{lucka2018risk}). In \cite{blanchard2012discrepancy}, it was first shown
 how to obtain optimal convergence in $L^2$ under Gaussian white noise with a modified version of the discrepancy principle.\\
 Another approach is to transfer results from the classical deterministic theory using the Ky-Fan metric, which metrises convergence in probability. In  
(\cite{hofinger2006ill},\cite{gerth2017lifting}) it is shown, how to obtain convergence if one knows the Ky-Fan distance between the measurements and the true data.
Aspects of the Bakushinskii veto \cite{bakushinskii1984remarks} for stochastic inverse problems are discussed in 
(\cite{bauer2008regularization},\cite{becker2011regularization},\cite{werner2018adaptivity}) under assumptions for the noise distribution. In particular,
\cite{becker2011regularization} gives an explicit non trivial example for a convergent regularisation, without knowing the exact error level, under Gaussian white
noise. We extent this to arbitrary distributions here, if one has multiple measurements.\\
 In the articles mentioned above, the error is usually modelled as a Hilbert space process (such as white noise), thus it is impossible to determine the  regularisation parameter
 directly through the discrepancy principle. This is in contrast to our, more classic error model,
 where the measurement is an element of the Hilbert space itself. Under the popular assumption that the operator $K$ is Hilbert-Schmidt, one could in principle extend our results to a general Hilbert space process error model (considering the symmetrised equation $K^*K\hat{x}=K^*\hat{y}$ instead of $K\hat{x}=\hat{y}$, as it is done for example in \cite{blanchard2012discrepancy}). But we will
 postpone the discussion of the white noise case to a follow up paper.\\
To summarise the connection to the Bakushinskii veto let us state the following. The Bakushinskii veto states that the inverse problem can only be solved with a 
deterministic regularisation, if the noise level of the data is known. In this article we show, that if one has access to multiple i.i.d. measurements of an 
unkown distribution, one may use as data the average together with the estimated noise level and one obtains the optimal deterministic rate with high probability, as the number of 
measurements tends to infinity. That is one can estimate the error from the data. Finally, the measurements potentially contain more information, which is not used
here. For example one could estimate the whole covariance structure of one measurement and use this to rescale the measurements and the operator, eventually increasing the relative smoothness of the data. Also one could directly
regularise the non-averaged measurements.\\
In the following section we apply our approach to a priori regularisations and in the main part we consider the widely used discrepancy principle,
which is known to work optimal in the classic deterministic theory. After that we quickly show how to choose $\delta_n^{est}$ to obtain almost sure
convergence and we compare the methods  numerically.



\section{A priori regularisation}
 We use the usual definition that
  $R_\alpha:\mathcal{Y}\to\mathcal{X}$ is called a linear regularisation, if $R_\alpha$ is
   a bounded linear operator for all $\alpha>0$ and if $R_{\alpha}y\to K^+y$ for $\alpha\to 0$
for all $y\in\mathcal{D}(K^+)$. A regularisation method is a combination of a regularisation  and a parameter choice strategy $\alpha : \mathbb{R}^+ \times \mathcal{Y} \to \mathbb{R}^+$, such that 
$R_{\alpha(\delta,y^{\delta})}y^{\delta} \to K^+y$ for $\delta\to0$, for all $y \in \mathcal{D}(K^+)$ and for all
 $(y^{\delta})_{\delta > 0}\subset \mathcal{Y}$ with $\| y^{\delta} - y \| \le \delta$. The method is called a priori, if the parameter choice does not depend
 on the data, that is if $\alpha(\delta,y)=\alpha(\delta)$.\\
The measurements can be 
formally modelled as realisations of an independent and identically distributed sequence $Y_1,Y_2,...: \Omega \to \mathcal{Y}$ of random variables with values in $\mathcal{Y}$, such that $\E Y_1 =\hat{y}\in \mathcal{D}(K^+)$. Moreover, we
require that $0<\E \| Y_1 \|^2 < \infty$, that is the measurements are (almost surely) in the Hilbert space.\\
In the following we apply the above approach to a priori parameter choice strategies $\alpha(y^{\delta},\delta)=\alpha(\delta)$. We restrict to $\delta_n^{est}=1/\sqrt{n}$ here, that is we do not estimate the variance here (otherwise the parameter choice would depend on the data). Since then $\delta_n^{est}$ and hence $\alpha(\delta_n^{est})$ are deterministic, the situation is very easy here and the results are not surprising (see Remark \ref{remark.2.1}).\\

\begin{theorem}[Convergence of a priori regularisation]\label{apriori}
  Assume that $K:\mathcal{X}\to \mathcal{Y}$ is a bounded linear operator with non-closed range between Hilbert spaces and that $Y_1,Y_2,...$ are i.i.d. $\mathcal{Y}-$valued random variables which fullfill $\E[ Y_1] = \hat{y}\in \mathcal{D}(K^+)$ and $0<\E\lVert Y_1\rVert^2<\infty$.
  Take an a priori regularisation scheme, with 
 $\alpha(\delta) \stackrel{\delta \to 0}{\longrightarrow} 0$ and $\| R_{\alpha(\delta)} \| \delta \stackrel{\delta \to 0}{\longrightarrow} 0$. Set
  $\bar{Y}_n:= \sum_{i\le n} Y_i/n$ and $\delta_n^{est}:=n^{-1/2}$.  Then $\lim_{n\to\infty}\E\lVert R_{\alpha(\delta_n^{est})} \bar{Y}_n -K^+\hat{y}\rVert^2 =0$.
\end{theorem}

\begin{proof}
Because of linearity, $\E \lek R_{\alpha} Y_1 \rek = R_{\alpha}\E \lek Y_1\rek= R_{\alpha}\hat{y}$ and thus by (\ref{phyth})
 
 \begin{align*}\label{apriorieq1}
  \E\| R_{\alpha} \bar{Y}_n - R_{\alpha} \hat{y}\|^2 &= \frac{1}{n^2}\E\left\lVert \sum_{i=1}^n R_{\alpha}\left(Y_i-\hat{y}\right)\right\rVert^2 = \frac{\E\| R_{\alpha}Y_1 - R_{\alpha}\hat{y}\|^2}{n},
 \end{align*}
\noindent since $R_{\alpha}Y_i \in \mathcal{R}(K^*)$ where the latter is separable.  Therefore, by the bias-variance-decomposition,
 \begin{align*}
  \E\| R_{\alpha(\delta_n^{est})}\bar{Y}_n - K^+\hat{y}\|^2 &= \E\| R_{\alpha(\delta_n^{est})}\bar{Y}_n - R_{\alpha(\delta_n^{est})} \hat{y} + R_{\alpha(\delta_n^{est})}\hat{y} - K^+\hat{y}\|^2\\
                                                               &= \E\| R_{\alpha(\delta_n^{est})}\bar{Y}_n - R_{\alpha(\delta_n^{est})} \hat{y}\|^2  + \| R_{\alpha(\delta_n^{est})}\hat{y} - K^+\hat{y}\|^2\\
                                                               &= \frac{\E\| R_{\alpha(\delta_n^{est})}Y_1 - R_{\alpha(\delta_n^{est})}\hat{y}\|^2}{n} + \| R_{\alpha(\delta_n^{est})}\hat{y} - K^+\hat{y}\|^2\\
                                                               &\le \frac{\| R_{\alpha(\delta_n^{est})}\|^2}{n} \E\| Y_1 - \hat{y}\|^2 + \| R_{\alpha(\delta_n^{est})}\hat{y} - K^+\hat{y}\|^2\\
                                                               &= \| R_{\alpha(\delta_n^{est})}\|^2{\delta_n^{est}}^2 \E\| Y_1 - \hat{y}\|^2 + \| R_{\alpha(\delta_n^{est})}\hat{y} - K^+\hat{y}\|^2\\
                                                               &\to 0 \qquad  \mbox{for}\quad n\to\infty.
 \end{align*}
\hfill$\qed$
\end{proof}

\noindent As in the deterministic case, under additional source conditions we can prove convergence rates.
We restrict to regularisations  $R_\alpha:=F_{\alpha}\left(K^*K\right)K^*$ defined via the spectral decomposition (see \cite{engl1996regularization}) with the following assumptions for the generating filter.
\begin{assumptions}\label{assumptions1}
$(F_{\alpha})_{\alpha>0}$ is a regularising filter, i.e. a family of piecewise continuous real valued 
functions on $[0,\lVert K\rVert^2]$, continuous from the right, with $\lim_{\alpha\to0}F_{\alpha}(\lambda)=\frac{1}{\lambda}$ for all $\lambda\in (0,\lVert K \rVert^2]$  and 
$\lambda F_{\alpha}(\lambda)\le C_R$ for all $\alpha>0$ and all $\lambda\in \left(0,\lVert K\rVert^2\right]$, where $C_R>0$ is some constant. Moreover, it has
qualification $\nu_0>0$, i.e. $\nu_0$ is maximal such that for all $\nu \in[0,\nu_0]$ there exists a constant $C_{\nu}>0$ with
\begin{equation*}
\sup_{\lambda\in(0,\lVert K\rVert^2]}\lambda^{\nu/2}| 1 -\lambda F_{\alpha}(\lambda)|\le C_{\nu}\alpha^{\nu/2}.
\end{equation*}
Finally, there is a constant $C_F>0$ such that 
$|F_{\alpha}(\lambda)|\le C_F/\alpha$ for all $0<\lambda\le \|K\|^2$.
\end{assumptions}

\begin{remark}
The generating filter of the following regularisation methods fullfill the Assumption \ref{assumptions1}:
\begin{enumerate}
\item  Tikhonov regularisation (qualification $2$)
\item $n$-times iterated Tikhonov regularisation (qualification $2n$),
\item truncated singular value regularisation (infinite qualification),
\item Landweber iteration (infinite qualification).
\end{enumerate}

\end{remark}

\begin{theorem}[Rate of convergence of aprioi regularisation]\label{aprioriws}
 Assume that $K:\mathcal{X}\to \mathcal{Y}$ is a bounded linear operator with non-closed range between  Hilbert spaces and that $Y_1,Y_2,...$ are i.i.d. $\mathcal{Y}-$valued
 random variables which fullfill $\E[ Y_1] = \hat{y}\in \mathcal{D}(K^+)$ and $0<\E\lVert Y_1\rVert^2<\infty$. Let $R_{\alpha}$ be induced by a filter
 fullfilling Assumption \ref{assumptions1}. Set $\bar{Y}_n:= \sum_{i\le n} Y_i/n$ and $\delta_n^{est}=n^{-1/2}$. Assume that for $0<\nu\le \nu_0$ and $\rho>0$ we have
 that $K^+\hat{y}=(K^*K)^{\nu/2}w$ for some $w\in\mathcal{X}$ with $\| w \| \le \rho$. Then if for constants $0<c<C$,
 \begin{equation*}
  c \left(\frac{\delta_n^{est}}{\rho}\right)^\frac{2}{\nu+1} \le \alpha(\delta_n^{est}) \le C \left(\frac{\delta_n^{est}}{\rho}\right)^\frac{2}{\nu+1},
 \end{equation*}
  \noindent we have that $\sqrt{\E\| R_{\alpha(\delta_n^{est})}\bar{Y}_n - K^+\hat{y} \|^2} \le C' {\delta_n^{est}}^\frac{\nu}{\nu+1} \rho^\frac{1}{\nu+1} = \mathcal{O}(n^{-\frac{\nu}{2(\nu+1)}}) $ for some constant $C'>0$.
\end{theorem}

\begin{proof}

 \noindent We proceed similiary to the proof of Theorem \ref{apriori}, using additionally Proposition \ref{propass1} of section \ref{sec:4}.
 
 \begin{align*}
  \E\| R_{\alpha(\delta_n^{est})}\bar{Y}_n - K^+\hat{y}\|^2 &= \E\| R_{\alpha(\delta_n^{est})}\bar{Y}_n - R_{\alpha(\delta_n^{est})} \hat{y}\|^2  + \| R_{\alpha(\delta_n^{est})}\hat{y} - K^+\hat{y}\|^2\\
                                                            &\le \| R_{\alpha(\delta_n^{est})}\|^2{\delta_n^{est}}^2 \E\| Y_1 - \hat{y}\|^2 + \| R_{\alpha(\delta_n^{est})}\hat{y} - K^+\hat{y}\|^2\\
                                                                               &\le  C_RC_F\E\| Y_1 - \hat{y}\|^2 \frac{{\delta_n^{est}}^2}{\alpha(\delta_n^{est})} + C_{\nu}^2 \rho^2 \alpha(\delta_n^{est})^{\nu}\\
                                                                               &\le  \frac{C_RC_F\E\| Y_1 - \hat{y}\|^2}{c}  {\delta_n^{est}}^\frac{-2}{\nu+1} \rho^\frac{2}{\nu+1} {\delta_n^{est}}^2  \\
                                                                               &~~~~~~~~~~~~~+ C_{\nu}^2 C^{\nu} {\delta_n^{est}}^\frac{2\nu}{\nu+1} \rho^\frac{-2\nu}{\nu+1} \rho^2\\
                                                                               &\le C'^2 {\delta_n^{est}}^\frac{2\nu}{\nu+1} \rho^\frac{2}{\nu+1}.
 \end{align*}
\hfill$\qed$
\end{proof}

\begin{remark}\label{remark.2.1}
For separable Hilbert spaces one could alternatively argue as follows: The spaces $\mathcal{X}':=L^2(\Omega,\mathcal{X})=\{X:\Omega \to \mathcal{X}:\E\lVert X\rVert^2<\infty\}$ and $\mathcal{Y}':=L^2(\Omega,\mathcal{Y})$ are also Hilbert spaces, with scalar products $(X,\tilde{X})_{\mathcal{X}'}:=\sqrt{\E(X,\tilde{X})_{\mathcal{X}}}$ and $(\cdot,\cdot)_{\mathcal{Y}'}$ defined similary. Then $K:\mathcal{X}\to\mathcal{Y}$ induces naturally a bounded linear operator $K':\mathcal{X}'\to\mathcal{Y}',X\mapsto KX$. Clearly we have that $\hat{y}\in \mathcal{Y}'$, and  $(\bar{Y}_n)_n$ is a sequence in $\mathcal{Y}'$ which fullfills 
\begin{equation*}
\lVert \bar{Y}_n-\hat{y}\rVert_{\mathcal{Y}'}:=\sqrt{(\bar{Y}_n-\hat{y},\bar{Y}_n-\hat{y})_{\mathcal{Y}'}} = \sqrt{\frac{\E\lVert Y_1-\hat{y}\rVert^2}{n}}=\sqrt{\E\lVert Y_1-\hat{y}\rVert^2}\delta_n^{est}
\end{equation*}
and we can use the classic deterministic results for $K':\mathcal{X}'\to\mathcal{Y}'$ and $\bar{Y}_n$ and $\delta_n^{est}$.

\end{remark}

\section{The discrepancy principle}

In this section we restrict to compact operators with dense range. Note that then $\mathcal{Y}=\overline{\mathcal{R}(K)}$ will  be automatically separable.
In practice the above parameter choice strategies are of limited interest, since they require the knowledge of the abstract smoothness parameters $\nu$ and
$\rho$. The classical discrepancy principle would be to choose $\alpha_n$ such that

\begin{equation}\label{eq.4.1}
 \| (KR_{\alpha_n}-Id)\bar{Y}_n\| \approx \delta_n^{true} =  \| \bar{Y}_n-\hat{y}\|,
\end{equation}

\noindent which is not possible, because of the unknown $\delta_n^{true}$. So we replace it with our estimator $\delta_n^{est}$ and implement the discrepancy principle via Algorithm 1 with or without the optional emergency stop.

  \begin{algorithm}\label{algorithm1}
 \caption{Discrepancy principle with estimated data error (optional: with emergency stop)}
\begin{algorithmic}[1]
\STATE Given measurements $Y_1,...,Y_n$;
\STATE Set $\bar{Y}_n:=\sum_{i\le n}Y_i/n$ and $\delta_n^{est}=1/\sqrt{n}$ or $\delta^{est}_n=\sqrt{\sum_{i\le n}\| Y_i -\bar{Y}_n\|^2/(n-1)}/\sqrt{n}$.
 \STATE Choose a $q\in (0,1)$. 
\STATE $k=0$;
  \WHILE{$\| (KR_{q^k}-Id)\bar{Y}_n\| > \delta_n^{est}$ (optional: and $q^k>1/n$)}
  \STATE $k=k+1$;
  \ENDWHILE
  \STATE $\alpha_n=q^k$;
\end{algorithmic}
\end{algorithm}

\begin{remark}
 To our knowledge, the idea of an emergency stop first appeared in \cite{blanchard2012discrepancy}. It provides a deterministic lower bound for
 the regularisation parameter, which may avoid overfitting. We use an elementary form of an emergency stop here, which does not require the knowledge of the singular value decomposition
 of $K$. It would be interesting to see, how more sophisticated versions of the emergency stop worked here, which is not clear to us since in our general setting
 we cannot rely
 on the concentration properties of Gaussian noise.
\end{remark}


\noindent Algorithm 1 will terminate, if we use the emergency stop. Otherwise, we can guarantee that Algorithm 1 terminates, if $K$ has dense image (or equivalently, if $K^*$ is injective) and if $\delta_n^{est}>0$. This is because then $\lim_{\alpha\to 0} KR_{\alpha}=P_{\overline{\mathcal{R}(K)}}=Id$ pointwise, so $\lVert (KR_{q^k}-Id)\bar{Y}_n\rVert < \delta_n^{est}$ for $k$ large enough . If we decided to use the sample variance, it may happen that $\delta_n^{est}=0$. But assuming $\E\lVert Y_1-\hat{y}\rVert^2>0$, it follows that $\mathbb{P}\left(\delta_n^{est}=0\right)=\mathbb{P}\left(Y_1=...=Y_n\right)\to 0$ for $n\to\infty$ (with exponential rate). If the distribution of $Y_1$ posseses a density (with respect to the Gaussian measure for example), then actually $\mathbb{P}(Y_1=...=Y_n)=0$ for all $n\in \mathbb{N}$.\\
 Unlike in the previous section, here the $L^2$ error will not converge in general, even if $Y_1$ has a density. The regularisation parameter $\alpha_n$ is now random, since it depends on the potentially bad random data. With a diminishing probability $p$ we are underestimating the data error 
significantly, and thus the discrepancy principle gives a too small $\alpha$ and we still have $p\| R_{\alpha}\| \gg 1$ in such a case.\\
In the following we will need the singular value decomposition of the compact operator $K$ with dense range (see \cite{cavalier2011inverse}): there exists a monotone sequence
$\lVert K \rVert =\sigma_1\ge \sigma_2 \ge ...>0$ with  $\sigma_l{\to}0$ for $l\to\infty$. Moreover there are families of orthonormal vectors 
 $(u_l)_{l\in\mathbb{N}}$ and $(v_l)_{l\in\mathbb{N}}$ with $span( u_l:l\in  \mathbb{N})=\mathcal{Y}$, $span(v_l:l\in\mathbb{N})= \mathcal{N}(K)^\bot$ 
 such that $Kv_l=\sigma_lv_l$ and  $K^*u_l=\sigma_lv_l$.
\subsection{A counter example for convergence}\label{counter}

We now show that a naive use of the discrepancy principle, as implemented in Algorithm 1 without emergency stop, may fail to converge in $L^2$.
To simplify calculations we pick Gaussian noise and the truncated singular value regularisation and we set $\delta_n^{est}=1/\sqrt{n}$.
We choose $\mathcal{X}:=l^2(\mathbb{N})$ with the standard basis $\{u_k:=(0,...,0,1,0,...)\}$ and consider the diagonal operator 

\begin{equation*}
K:l^2(\mathbb{N})\to l^2(\mathbb{N}),\quad u_l \mapsto \left(\frac{1}{100}\right)^\frac{l}{2} u_l
\end{equation*}

\noindent with $\hat{x}=0=\hat{y}=K\hat{x}$. Hence the $\sigma_l=(1/100)^\frac{l}{2}$ are the eigenvalues of $K$ and 

\begin{equation*}
R_{\alpha}:l^2(\mathbb{N})\to l^2(\mathbb{N}), \quad y \mapsto \sum_{l:\sigma_l^2\ge\alpha} \sigma_l^{-1}(y,u_l)u_l.
\end{equation*}
 \noindent We assume that the noise is distributed along $y:= \sum_{l\ge 2} 1/\sqrt{l(l-1)} u_l$, so we have that $\sum_{l> n} (y,u_l)^2=1/n$
  and thus $y\in l^2(\mathbb{N})$. That is we set
$\bar{Y}_n:=\sum_{i\le n} Y_i = \sum_{i\le n} Z_iy$, where $Z_i$ are i.i.d. standard Gaussians. We define $\Omega_n:=\{Z_i\ge 1, i=1...n\}$, a (very unlikely) event
on which we significantly underestimate the true data error. We get that
$\mathbb{P}(\Omega_n):=\mathbb{P}(Z_1\ge 1)^n\ge 1/10^n$. Moreover, by the definition of the discrepancy principle

\begin{align*}
 \frac{1}{n} \chi_{\Omega_n}={\delta_n^{est}}^2 \chi_{\Omega_n} &\ge \| (KR_{\alpha_n}-Id)\bar{Y}_n\|^2 \chi_{\Omega_n} = |\bar{Z}_n|^2 \| (KR_{\alpha_n}-Id)y\|^2 \chi_{\Omega_n}\\
                             &\ge \| (KR_{\alpha_n}-Id)y\|^2\chi_{\Omega_n}\\
                             &=\sum_{l:\sigma_l^2<\alpha_n} (y,u_l)^2 \chi_{\Omega_n}= \sum_{l:(1/100)^i<\alpha_n} (y,u_l)^2\chi_{\Omega_n}\\
                             &= \sum_{l> \frac{\log(\alpha_n)}{\log(1/100)}}(y,u_l)^2 \chi_{\Omega_n}\ge \frac{\log(1/100)}{\log(\alpha_n)}\chi_{\Omega_n}\\
                             \Longrightarrow \alpha_n \chi_{\Omega_n} &< \frac{1}{100^n}.
\end{align*}

\noindent It follows that

\begin{align*} 
\E\| R_{\alpha_n} \bar{Y}_n - K^+ \hat{y}\|^2 &= \E\| R_{\alpha_n}\bar{Y}_n\|^2 \ge \E\| R_{\alpha_n} \bar{Y}_n \chi_{\Omega_n}\|^2\\
                                                     &=  \E\left[\bar{Z}_n^2\| R_{\alpha_n}y \chi_{\Omega_n}\|\right]^2\ge \E\lVert R_{1/100^n} y \chi_{\Omega_n}\rVert^2\\
                                                     &\ge \sum_{l:\sigma_i^{2}\ge 1/100^n} \sigma_l^{-2}(y,u_l)^2 \mathbb{P}(\Omega_n) \ge\frac{1}{10^n} \sum_{l\le n} \sigma_l^{-2}(y,u_l)^2\\
                                                     &\ge \frac{1}{10^n} 100^n (y,u_n)^2 = \frac{10^n}{n(n-1)}\to \infty.
\end{align*}

\noindent That is the probability of the events $\Omega_n$ is not small enough to compensate the huge error we have on these events, so in the end $\E\lVert R_{\alpha_n}\bar{Y}_n-K^+\hat{y}\rVert^2\to \infty$ for $n\to\infty$.

 \subsection{Convergence in probability of the discrepancy principle}

In this section we show, that the discrepancy principle yields convergence in probability, matching asymptotically the optimal deterministic rate.
The proofs of the Theorems \ref{thwos} and \ref{thws} and of Corollary \ref{cordisc} are given in the following section.

\begin{theorem}[Convergence of the discrepancy principle]\label{thwos}
  Assume that $K$ is a compact operator with dense range between  Hilbert spaces $\mathcal{X}$ and $\mathcal{Y}$ and that $Y_1,Y_2,...$ are i.i.d. $\mathcal{Y}-$valued
  random variables with $\E Y_1=\hat{y}\in \mathcal{R}(K)$ and $0<\E\| Y_1-\hat{y}\|^2 < \infty$. Let $R_{\alpha}$ be induced by a filter fullfilling Assumption \ref{assumptions1} with $\nu_0>1$. Applying Algorithm 1 with or without the emergency stop 
   yields a sequence $(\alpha_n)_n$. Then
  we have that for all $\varepsilon> 0$
  
  \begin{equation*}
   \mathbb{P}\left(\| R_{\alpha_n}\bar{Y}_n - K^+\hat{y}\| \le \varepsilon \right)\stackrel{n\to\infty}{\longrightarrow}1,
  \end{equation*}

  \noindent i.e. $R_{\alpha_n}\bar{Y}_n \prob K^+\hat{y}$.
\end{theorem}

\begin{remark}
 If one tried to argue as in Remark $1$ to show $L^2$ convergence one would have to determine the regularisation parameter not as given by equation (\ref{eq.4.1}), but such that $\E\| (KR_{\alpha}-Id)\bar{Y}_n\|^2 \approx \delta_n^{est}$, which is not practicable since we cannot calculate the expectation on the left hand side.
\end{remark}

\noindent The popularity of the discrepancy principles is a result of the fact that it guarantees optimal convergence rates under an additional source condition:  Assuming that there is a $0<\nu\le \nu_0-1$ (where $\nu_0$ is the qualification of the chosen regularisation method) such that $K^+\hat{y}=\left(K^*K\right)^\frac{\nu}{2}w$ for a $w\in \mathcal{X}$ with $\lVert w \rVert \le \rho$, then 

\begin{equation}\label{deterministicbound}
 \sup_{y^{\delta}:\| y^{\delta}-\hat{y}\| \le \delta}\| R_{\alpha(y^{\delta},\delta)}y^{\delta}-K^+\hat{y}\| \le C \rho^\frac{1}{\nu+1}\delta^\frac{\nu}{\nu+1}
\end{equation}

\noindent for some constant $C>0$. The next theorem shows a concentration result for the discrepancy principle as implemented in Algorithm 1, with a bound similiar to (\ref{deterministicbound}).
 \begin{theorem}[Rate of convergence of the discrepancy principle]\label{thws}
  Assume that $K$ is a compact operator with dense range between Hilbert spaces $\mathcal{X}$ and $\mathcal{Y}$. Moreover, $Y_1,Y_2,...$ are i.i.d. $\mathcal{Y}-$valued
  random variables with $\E Y_1=\hat{y}\in \mathcal{R}(K)$ and $0<\E\| Y_1-\hat{y}\|^2 < \infty$.  Let $R_{\alpha}$ be induced by a filter fullfilling Assumption \ref{assumptions1} with $\nu_0>1$.
  Moreover, assume that there is a $0<\nu\le \nu_0-1$ and a $\rho>0$ such that $K^+\hat{y}=(K^*K)^{\nu/2}w$ for some $w\in\mathcal{X}$ with $\| w \| \le \rho$. Applying Algorithm 1 with or without the emergency stop yields a sequence $(\alpha_n)_{n\in\mathbb{N}}$. Then there is a constant $L$,
  such that
  
  \begin{equation*}
   \mathbb{P}\left(\| R_{\alpha_n}\bar{Y}_n - K^+\hat{y}\| \le L \rho^\frac{1}{\nu+1} \max\left\{{\delta_n^{est}}^\frac{\nu}{\nu+1},{\delta_n^{true}}^\frac{\nu}{\nu+1}\left(\delta_n^{true}/\delta_n^{est}\right)^\frac{1}{\nu+1}\right\}\right)\stackrel{n\to\infty}{\longrightarrow}1.
  \end{equation*}

 \end{theorem}

\noindent We deduce a deterministic bound for $\lVert R_{\alpha_n}\bar{Y_n}-K^+\hat{y}\rVert$ (for $n$ large). 

\begin{corollary}\label{corthws}
Under the assumptions of Theorem \ref{thws}, for all $\varepsilon>0$ it holds that
   \begin{equation*}
   \mathbb{P}\left(\| R_{\alpha_n}\bar{Y}_n - K^+\hat{y}\| \le \rho^\frac{1}{\nu+1} \left(\frac{1}{\sqrt{n}}\right)^{\frac{\nu}{\nu+1}-\varepsilon}\right)\stackrel{n\to\infty}{\longrightarrow}1.
  \end{equation*}

\end{corollary} 
 
 \begin{proof}[Corollary \ref{corthws}]
By the second assertion in  Lemma $\ref{centlemma}$ and Markov's inequality, for any $c,\varepsilon>0$,
\begin{equation*}
\lim_{n\to\infty}\mathbb{P}\left(\delta_n^{est},\delta_n^{true}\le cn^{-\frac{1}{2}+\varepsilon}\right)=1.
\end{equation*}
\qed
 
 \end{proof}

\noindent  The ad hoc emergency stop $\alpha_n>1/n$, additionally assures, that the $L^2$ error will not explode
(unlike in the counter example of the previous subsection). Under the assumption that $\E \lVert Y_1-\hat{y}\rVert^4<\infty$, one can guarantee, 
that the $L^2$ error will converge.
 
\begin{corollary}\label{cordisc}
Under the assumptions of Theorem \ref{thwos}, consider the sequence $\alpha_n$ determined by Algorithm 1 with emergency stop. Then there is a constant $C$ such that
$\E \lVert R_{\alpha_n}\bar{Y}_n-K^+\hat{y}\rVert^2\le C$ for all $n\in\mathbb{N}$. If additionally $\E \lVert Y_1-\hat{y}\rVert^4<\infty$,
then it holds that $\E\lVert R_{\alpha_n} \bar{Y}_n -K^+\hat{y}\rVert^2 \to0$ for $n\to\infty$.
\end{corollary}

\label{remarkdisc}

\subsection{Almost sure convergence}
The results so far delievered either convergence in probability or convergence in $L^2$. We give a short remark how one can obtain almost sure convergence. Roughly speaking, one has to multiply a $\sqrt{\log\log n}$ term to $\delta_n^{est}$. This is a simple consequence of the following theorem

\begin{theorem}[Law of the iterated logarithm]
Assume that $Y_1,Y_2,...$ is an i.i.d sequence with values in some seperable Hilbert space 
$\mathcal{Y}$. Moreover, assume that $\E Y_1 = 0$ and $\E\| Y_1\|^2<\infty$. Then we have that

\begin{equation*}
\mathbb{P}\left(\limsup_{n\to\infty} \frac{\|\sum_{i\le n} Y_i \|}{\sqrt{2 \E\| Y_1 \|^2n\log\log n}}\le 1\right) = 1.
\end{equation*}

\end{theorem}

\begin{proof}
This is a simple consequence of Corollary 8.8 in \cite{ledoux1991probability}.
\end{proof}

\noindent So if $\E Y_1 = \hat{y} \in  \mathcal{Y}$ we have for $\delta_n^{true}=\lVert \bar{Y}_n-\hat{y}\rVert$

\begin{equation*}
\mathbb{P}\left(\limsup_{n\to\infty} \frac{\sqrt{n}\delta_n^{true}}{\sqrt{2\E\lVert Y_1-\hat{y}\rVert^2\log\log n}} \le  1\right)=1,
\end{equation*}

\noindent that is, with probability $1$ it holds that $\delta_n^{true}\le\sqrt{\frac{2\E\lVert Y_1-\hat{y}\rVert^2\log\log n}{n}}$ for $n$ large enough. Consequently, for some $\tau>1$ the estimator should be

\begin{equation*}
\delta_n^{est}:= \tau s_n \sqrt{\frac{2\log\log n}{n}},
\end{equation*}

\noindent where $ s_n$ is the square root of the sample variance. Since $\mathbb{P}(\lim_{n\to\infty} s_n^2=\E\lVert Y_1-\hat{y}\rVert^2)=1$ and $\tau>1$ it holds that $\sqrt{\E\lVert Y_1-\hat{y}\rVert}\le \tau s_n$ for $n$ large enough with probability $1$ and
thus $\delta_n^{true}\le \delta_n^{est}$ for $n$ large enough with probability $1$. In other words, there is an event $\Omega_0 \subset \Omega$ with $\mathbb{P}(\Omega_0)=1$ such that for any $\omega \in \Omega_0$ there is a $N(\omega)\in \mathbb{N}$ with
$\delta_n^{true}(\omega)\le\delta_n^{est}(\omega)$ for all $n\ge N(\omega)$. So we can use $\bar{Y}_n$ and $\delta_n^{est}$
together with any deterministic regularisation method to get almost sure convergence.

\section{Proofs of Theorem \ref{thwos} and \ref{thws}}\label{sec:4}
\subsection{Proofs without emergency case}
We will multiple times use the Pythagorean theorem for independent separable Hil-\\
bert
 space valued random variables $Z_i$  with $\E \lVert Z_i\rVert^2<\infty$ and $\E Z_i=0$,

\begin{equation}\label{phyth}
\E \left\lVert \sum_{i=1}^n Z_i \right\rVert^2 = \sum_{i=1}^n \sum_{l,l'=1}^\infty \E\left[(Z_i,e_l)(Z_i,e_{l'})\right] = \sum_{i=1}^n\E\left[\sum_{j=1}^\infty(Z_i,e_j)^2\right]=\sum_{i=1}^n \E\left\lVert Z_i \right\rVert^2,
\end{equation}
\noindent where $(e_l)_{l\in\mathbb{N}}$ is an orthonormal basis. Based on this, the central ingridient will be the following lemma, which strengthens the pointwise worst case error bound $\lVert (KR_{\alpha}-Id)(\bar{Y}_n-\hat{y})\rVert \le C_0 \delta_n^{true}$ in some sense.

\begin{lemma}\label{centlemma}
For all $\varepsilon>0$ and (deterministic) sequences $(q_n)_{n\in\mathbb{N}}$ with $q_n>0$ and $\lim_{n\to\infty}q_n=0$, it holds that

\begin{equation*}
\mathbb{P}\left(\lVert (KR_{q_n}-Id)(\bar{Y}_n-\hat{y})\rVert \ge \varepsilon/\sqrt{n}\right) \to 0
\end{equation*}

\noindent and 

\begin{equation*}
\mathbb{P}\left(|\sqrt{n}\delta_n^{est}-\gamma|\ge\varepsilon\right)\to 0
\end{equation*}

\noindent for $n\to\infty$, where $\gamma=1$ or $\gamma=\sqrt{\E\lVert Y_1-\hat{y}\rVert^2}$, depending on if we used the sample variance or not.
\end{lemma}

\begin{proof}
By Tschebyscheff's inequality and (\ref{phyth})

\begin{align*}
\mathbb{P}\left(\lVert (KR_{q_n}-Id)(\bar{Y}_n-\hat{y})\rVert \ge \varepsilon/\sqrt{n}\right) &\le \frac{n}{\varepsilon^2}\E\lVert (KR_{q_n}-Id)(\bar{Y}_n-\hat{y})\rVert^2\\
          &= \frac{1}{\varepsilon^2}\E\lVert(KR_{q_n}-Id)(Y_1-\hat{y})\rVert^2.
\end{align*}

\noindent  Since $K$ has dense range, $KR_{q_n}-Id$ converges to $0$ pointwise for $n\to\infty$ and it follows that $(KR_{q_n}-Id)(Y_1-\hat{y})$ also converges pointwise to $0$. By inequality (\ref{propeq2}) of Proposition \ref{propass1} below, $\lVert (KR_{q_n}-Id)(Y_1-\hat{y}) \rVert^2 \le C_0 \lVert Y_1-\hat{y}\rVert^2$, so
$\E\lVert(KR_{q_n}-Id)(Y_1-\hat{y})\rVert^2\to 0$ for $n\to\infty$ by the dominated convergence theorem.  The second assertion only needs a proof for $\gamma=\sqrt{\E\lVert Y_1-\hat{y}\rVert^2}$ and then

\begin{align*}
n {\delta_n^{est}}^2 = \frac{1}{n-1} \sum_{i=1}^n\lVert Y_i-\bar{Y}_n\rVert^2 &= \frac{n}{n-1} \left(\frac{1}{n}\sum_{i=1}^n\lVert Y_i\rVert^2- \lVert \bar{Y}_n\rVert^2\right) \\
&\to \E \lVert Y_1\rVert^2-\lVert \hat{y}\rVert^2 =\E\lVert Y_1-\hat{y}\rVert^2= \gamma^2
\end{align*}

\noindent almost surely (thus in particular in probability) for $n\to\infty$ by the strong law of large numbers (Corollary 7.10 in \cite{ledoux1991probability}) and the bias-variance-decomposition. Therefore $\sqrt{n}\delta_n^{est}\to \gamma$ in probability for $n\to\infty$.
\qed

\end{proof}

\noindent For convergence in probability it does not
matter how large the error is on sets with diminishing probability and with Lemma \ref{centlemma} we will show, that the probability of certain 'good events' is 1 in the limit of infinitely many measurements. \\
Define for $q\in(0,1)$ (as chosen in Algorithm 1)

\begin{align}\label{th4eq1}
\psi_q:~&\mathbb{R}^+ \to \left\{q^k~:~k\in \mathbb{N}_0\right\}\\\notag
       & \alpha\mapsto \max\left\{q^k: q^k\le \alpha\right\}.
\end{align}

\noindent So $\min(q\alpha,1)\le\psi_q(\alpha)\le \alpha$ and by definition, if $\lVert \left(KR_{\psi_q(\alpha)}-Id\right)\bar{Y}_n\rVert<\delta_n^{est}$, it holds that $\alpha_n\ge\min(q\alpha,1)$, where $\alpha_n$ is the output of Algorithm 1.\\
 We will also need some well known properties of regularisations defined by filters which fullfill Assumption \ref{assumptions1}. These are mostly easy modifications from \cite{engl1996regularization}.

\begin{proposition}\label{propass1}
The constants in the following are defined as in Assumption \ref{assumptions1}. We assume, that $K$ is bounded and linear with non-closed range.
Assume that $(R_{\alpha})_{\alpha>0}$ is induced by a regularising filter fullfilling $|F_{\alpha}(\lambda)|\le C_F/\alpha$ for all $0<\lambda\le \|K\|^2$.  Then

\begin{align}\label{propeq1}
\lVert R_{\alpha}\rVert &\le \sqrt{C_RC_F}/\sqrt{\alpha}\\\label{propeq2}
\lVert Id-KR_{\alpha}\rVert &\le C_0
\end{align}
 
\noindent for all $\alpha>0$, with $C_0\ge 1$. If moreover, the filter has qualification $\nu_0>0$ and there is a $w \in \mathcal{X}$ with  $\lVert w \rVert \le \rho$ such that $K^+\hat{y}=\left(K^*K\right)^\frac{\nu}{2}w$   for some $0<\nu\le\nu_0$, then
\begin{align}\label{propeq3}
\lVert R_{\alpha}\hat{y}-K^+\hat{y}\rVert &\le C_{\nu} \rho \alpha^{\nu/2}\\\label{propeq4}
\lVert R_{\alpha}\hat{y}-K^+\hat{y}\rVert &\le \lVert KR_{\alpha}\hat{y}-KK^+\hat{y}\rVert^\frac{\nu}{\nu+1}C_0^\frac{1}{\nu+1}\rho^\frac{1}{\nu+1}
\end{align}

\noindent for all $\alpha>0$. If additionally, $\nu_0\ge\nu+1>1$, then 

\begin{equation}\label{propeq5}
\lVert KR_{\alpha}\hat{y}-KK^+\hat{y}\rVert \le C_{\nu+1}\rho \alpha^\frac{\nu+1}{2}.
\end{equation}
 
 \noindent  Moreover, if $K$ is compact, than for all $x\in\mathcal{X}$ there is a function  $g:\mathbb{R}^+\to \mathbb{R}^+$ with $g(\alpha)\to\infty$ for $\alpha\to 0$, such that
 
 \begin{equation}\label{propeq6}
 \lim_{\alpha\to 0}\lVert(KR_{\psi_q\left(\alpha g(\alpha)\right)}-Id)Kx\rVert/\sqrt{\alpha} = 0,
 \end{equation}

\noindent where $\psi_q$ is given in (\ref{th4eq1}). 
 
\end{proposition}

\begin{proof}[Proposition \ref{propass1}]
(\ref{propeq1}) and (\ref{propeq4}) are shown in the proofs of Theorem 4.2 and Theorem 4.17 in \cite{engl1996regularization}. (\ref{propeq3}) and (\ref{propeq4}) are Theorem 4.3 in \cite{engl1996regularization}. (\ref{propeq2}) follows directly from Assumption \ref{assumptions1}.\\
For (\ref{propeq6}), let $x\in\mathcal{X}$ be fixed and set

\begin{equation*}
\tilde{g}(\alpha):=\sup\left\{ t>0~:~\left\lVert \left(KR_{\psi_q\left(\alpha t\right)}-Id\right)Kx\right\rVert/\sqrt{\alpha}\le t^{-1}\right\}.
\end{equation*}

\noindent W.l.o.g.  $\tilde{g}$ is finite for any $\alpha>0$.
Now we first show that 

\begin{equation}\label{proofprop1}
\lim_{\alpha\to0}\lVert (KR_{\alpha}-Id)Kx\rVert/\sqrt{\alpha}=0.
\end{equation}
\noindent We mimic the proof of Theorem 3.1.17 of \cite{nakamura2015inverse} and set $\varepsilon>0$. We  fix $L$, such that $C_1^2 \sum_{l=L+1}^\infty (\hat{x},v_j)^2<\varepsilon$. Then

\begin{align*}
& \lVert (KR_{\alpha}-Id)K\hat{x}\rVert^2/\alpha = \sum_{l=1}^\infty  \left(F_{\alpha}(\sigma_l^2)\sigma_l^2-1\right)^2 \frac{\sigma_l^2}{\alpha}(\hat{x},v_l)^2\\
  \le &\left(\sup_{\lambda>0} \lambda^\frac{\nu_0}{2}|F_{\alpha}(\lambda)\lambda-1|\right)^2\lVert \hat{x}\rVert^2\sum_{l=1}^L \frac{\sigma_l^{2(1-\nu_0)}}{\alpha}\\
  &\qquad + \left(\sup_{\lambda>0} \lambda^\frac{1}{2}|F_{\alpha}(\lambda)\lambda-1|\right)^2\frac{\sum_{l=L+1}^\infty (\hat{x},v_j)^2}{\alpha}\\
  \le &C_{\nu_0}^2L \sigma_L^{2(1-\nu_0)}\lVert \hat{x}\rVert^2\alpha^{\nu_0-1}+C_1^2 \sum_{l=L+1}^\infty (\hat{x},v_j)^2< 2 \varepsilon
\end{align*}

\noindent for all $\alpha<\left(\varepsilon^{-1} C_{\nu_0}^2L\sigma_L^{2(1-\nu_0)}\lVert \hat{x}\rVert^2\right)^{-\frac{1}{\nu_0-1}}$, therefore  $\lVert (KR_{\alpha}-Id)Kx\rVert/\sqrt{\alpha}=0$ for $\alpha\to 0$. So for any $t>0$
\begin{align*}
\left\lVert \left(KR_{\psi_q\left(\alpha t\right)}-Id\right)Kx\right\rVert/\sqrt{\alpha} &= \sqrt{\frac{\psi_q\left(\alpha t\right)}{\alpha}}\left\lVert \left(KR_{\psi_q\left(\alpha t\right)}-Id\right)Kx\right\rVert/\sqrt{\psi_q\left(\alpha t\right)}\le \frac{1}{t} 
\end{align*}
 for $\alpha$ small enough, because of (\ref{proofprop1}) and since $\psi_q(\alpha t)\le \alpha t$. So  $\tilde{g}(\alpha)\to\infty$ for $\alpha\to 0$ and by definition of $\tilde{g}$ the claim holds for $g(\alpha):=\tilde{g}(\alpha)-1$ ($g$ is well defined for $\alpha$ small enough).
\qed

\end{proof}

\begin{proof}[Theorem 4]
 Set $q_n:=\psi_q(b_n)$ where $b_n:=\left(\frac{1}{\rho}\frac{\gamma}{4C_{\nu+1}\sqrt{n}}\right)^\frac{2}{\nu+1}$ with $\gamma=1$ or $\gamma=\sqrt{\E \lVert Y_1-\hat{y}\rVert^2}$, depending on if we used the sample variance or not, and $\psi_q$ given in (\ref{th4eq1}). Define

\begin{equation}\label{th4omega1}
\Omega_n:=\Omega_n(q_n,\gamma):=\left\{ |\sqrt{n}\delta_n^{est}-\gamma|< \gamma/2~,~\lVert(KR_{q_n}-Id)(\bar{Y}_n-\hat{y})\rVert< \gamma/\sqrt{16n}\right\}.
\end{equation}

\noindent  Then by (\ref{propeq5}) and since $q_n\le b_n$,

\begin{align}\label{th4eq2}
\lVert (KR_{q_n}-Id)\bar{Y}_n\rVert\chi_{\Omega_n} &\le \lVert (KR_{q_n}-Id)\hat{y}\rVert\chi_{\Omega_n} + \lVert (KR_{q_n}-Id)(\bar{Y}_n-\hat{y})\rVert\chi_{\Omega_n}\\\notag
                    &\le C_{\nu+1}\rho b_n^\frac{\nu+1}{2}\chi_{\Omega_n} + \frac{\gamma}{4\sqrt{n}}\chi_{\Omega_n} =\frac{\gamma}{2\sqrt{n}}\chi_{\Omega_n}<\delta_n^{est}\chi_{\Omega_n},
\end{align}

\noindent so $\alpha_n \chi_{\Omega_n} \ge q b_n \chi_{\Omega_n} \ge q \left(\frac{\delta_n^{est}}{6C_{\nu+1}}\right)^\frac{2}{\nu+1}\chi_{\Omega_n}$ for $n$ large enough.
By (\ref{propeq4}), (\ref{propeq2}) and since $K$ has dense image,

\begin{align*}
 \| R_{\alpha_n}\hat{y} - K^+\hat{y}\|  &\le \lVert KR_{\alpha_n}\hat{y}-KK^+\hat{y}\rVert^\frac{\nu}{\nu+1}C_0^\frac{1}{\nu+1}\rho^\frac{1}{\nu+1}= \lVert KR_{\alpha_n}\hat{y}-\hat{y}\rVert^\frac{\nu}{\nu+1}C_0^\frac{1}{\nu+1}\rho^\frac{1}{\nu+1}\\
                                           &\le \left(\| (KR_{\alpha_n}-Id)\bar{Y}_n\| + \| (KR_{\alpha_n}-Id)(\hat{y} - \bar{Y}_n)\|\right)^\frac{\nu}{\nu+1}C_0^\frac{1}{\nu+1}\rho^\frac{1}{\nu+1}\\
                                           &\le \left(\delta_n^{est} + \| (KR_{\alpha_n}-Id)(\hat{y} - \bar{Y}_n)\|\right)^\frac{\nu}{\nu+1}C_0^\frac{1}{\nu+1}\rho^\frac{1}{\nu+1}\\
                                           &\le   \left(\delta_n^{est} + C_0\delta_n^{true}\right)^\frac{\nu}{\nu+1}C_0^\frac{1}{\nu+1}\rho^\frac{1}{\nu+1}\le \left(\delta_n^{est} + \delta_n^{true}\right)^\frac{\nu}{\nu+1}C_0\rho^\frac{1}{\nu+1} .                                     
\end{align*} 

\noindent Finally, 

\begin{align*}\notag
 &\| R_{\alpha_n}\bar{Y}_n - K^+\hat{y}\| \chi_{\Omega_n}\\\notag
 \le& \| R_{\alpha_n}\hat{y} - K^+\hat{y}\| \chi_{\tilde{\Omega}_n} + \| R_{\alpha_n}\bar{Y}_n - R_{\alpha_n}\hat{y}\|\chi_{\Omega_n}\\\notag
                                                                 \le& \left( \delta_n^{est}+\delta_n^{true} \right)^{\frac{\nu}{\nu+1}} C_0\rho^{\frac{1}{\nu+1}}\chi_{\Omega_n}   +\sqrt{C_RC_F} \frac{\delta_n^{true}}{\sqrt{\alpha_n}}\chi_{\Omega_n}\\
    \le &\left(2\max\left( \delta_n^{est},\delta_n^{true}\right) \right)^{\frac{\nu}{\nu+1}} C_0\rho^{\frac{1}{\nu+1}}   \chi_{\Omega_n}+\sqrt{C_RC_F} \rho^\frac{1}{\nu+1}\left(\frac{6C_{\nu+1}}{\delta_n^{est}}\right)^\frac{1}{\nu+1}\frac{\delta_n^{true}}{\sqrt{q}}\chi_{\Omega_n}\\                                                       
    \le& L\max\left\{{\delta_n^{est}}^\frac{\nu}{\nu+1},{\delta_n^{true}}^\frac{\nu}{\nu+1}\left(\frac{\delta_n^{true}}{\delta_n^{est}}\right)^\frac{1}{\nu+1}\right\},
\end{align*}

\noindent with $L:=2^\frac{\nu}{\nu+1}C_0\rho^\frac{1}{\nu+1}+\sqrt{C_RC_F/q}\left(6C_{\nu+1}\right)^\frac{1}{\nu+1}$ and the proof is finished, because $\mathbb{P}\left(\Omega_n\right)\to1$ for $n\to\infty$  by Lemma \ref{centlemma}.
\qed
\end{proof}

\begin{proof}[Proof of Theorem 3]
W.l.o.g. we may assume that there are arbitrarily large $l\in\mathbb{N}$ with $(\hat{y},u_l)\neq 0$, since otherwise we could apply Theorem $4$ with any $\nu>0$. Let $\varepsilon'>0$. Then there is a $L\in\mathbb{N}$ such that $(\hat{y},u_L)\neq 0$ and $\left(F_{q^k}(\sigma_L^2)\sigma_L^2-1\right)^2>1/2$ for all $k\in\mathbb{N}_0$ with $q^k\ge\varepsilon'$ (because the $F_{q^k}$ are bounded and $\sigma_l\to 0$ for $l\to\infty$). Set
\begin{equation}\label{th3omega0}
\Omega_n:=\left\{| \sqrt{n}\delta_n^{est}-\gamma|< \gamma~,~(\bar{Y}_n,u_L)^2\ge (\hat{y},u_L)^2/2\right\}.
\end{equation}

\noindent Then for $n\ge 16\gamma^2/(\hat{y},u_L)^2$, 

\begin{align*}
\delta_n^{est}\chi_{\Omega_n} &\le \frac{2\gamma}{\sqrt{n}}\chi_{\Omega_n} < \sqrt{\frac{(\hat{y},u_L)^2}{4}}\chi_{\Omega_n} \le \sqrt{\left(F_{q^k}(\sigma_L^2)\sigma_L^2-1\right)^2 (\bar{Y}_n,u_L)^2}\chi_{\Omega_n}\\
&\le \sqrt{ \sum_{l=1}^\infty \left(F_{q^k}(\sigma_l^2)\sigma_l^2-1\right)^2\left(\bar{Y}_n,u_l\right)^2}\chi_{\Omega_n} = \lVert (KR_{q^k}-Id)\bar{Y}_n\rVert\chi_{\Omega_n}
\end{align*}

\noindent for all $k\in\mathbb{N}_0$ with $q^k\ge \varepsilon'$. Thus for $\Omega_n$ given in (\ref{th3omega0})

\begin{equation}\label{p3eq1}
\lim_{n\to\infty}\mathbb{P}\left(\alpha_n\le \varepsilon'\right)\ge \lim_{n\to\infty}\mathbb{P}\left( \Omega_n\right)= 1
\end{equation}
\noindent  by Lemma \ref{centlemma} and since $(\bar{Y}_n,u_L)=\sum_{i=1}^n(Y_i,u_L)/n\to \E(Y_1,u_L)=(\hat{y},u_L)\neq0$ almost surely for $n\to\infty$.
Set $q_n:=\psi_q\left(b_n\right)$ with $b_n:=n^{-1}g(n^{-1})$ and  $g$ and $\psi_q$ given in (\ref{propeq6})  and (\ref{th4eq1}). Define

\begin{equation}\label{th3omega1}
\Omega_n:=\left\{| \sqrt{n}\delta_n^{est}-\gamma|< \gamma/2~,~\lVert (KR_{q_n}-Id)(\bar{Y}_n-\hat{y})\rVert< \gamma/4\sqrt{n}\right\}.
\end{equation}

\noindent Then for $n$ large enough (such that $\lVert (KR_{q_n}-Id)\hat{y}\rVert\sqrt{n}\le \gamma/4$, see (\ref{propeq6}) with $\alpha=n^{-1}$),

\begin{align}\label{th3eq1}
 \lVert (KR_{q_n}-Id)\bar{Y}_n\rVert\chi_{\Omega_n}
 &\le \frac{1}{\sqrt{n}}\sqrt{n}\lVert (KR_{q_n}-Id)\hat{y}\rVert\chi_{\Omega_n} + \lVert (KR_{q_n}-Id)(\bar{Y}_n-\hat{y})\rVert \chi_{\Omega_n}\\\notag
 \le &\frac{\gamma}{4\sqrt{n}}\chi_{\Omega_n}+ \frac{\gamma}{4\sqrt{n}}\chi_{\Omega_n} \le \frac{\gamma}{2\sqrt{n}}\chi_{\Omega_n}\le \delta_n^{est}\chi_{\Omega_n}.
\end{align}

\noindent That is $\alpha_n \chi_{\Omega_n}\ge q b_n \chi_{\Omega_n}\ge q n^{-1}g(n^{-1}) \chi_{\Omega_n}$ for $n$ large enough. Finally set

\begin{equation*}
\tilde{\Omega}_n :=\left\{ \delta_n^{true}\le \sqrt{\frac{\sqrt{g(n^{-1})}}{n}}~,~\lVert R_{\alpha_n}\hat{y}-K^+\hat{y}\rVert \le \frac{\varepsilon}{2}\right\}\cap \Omega_n,
\end{equation*}

\noindent with $\Omega_n$ given in (\ref{th3omega1}). So $\mathbb{P}\left(\tilde{\Omega}_n\right)\to 1$ for $n\to\infty$, since $\mathbb{P}\left(\delta_n^{true}\le \sqrt{\sqrt{g(n^{-1})}/n}\right)\to1$, 
because of $g(n^{-1})\to \infty$, $\mathbb{P}\left(\Omega_n\right)\to 1$ by 
Lemma \ref{centlemma} and $\mathbb{P}\left(\lVert R_{\alpha_n}\hat{y}-K^+\hat{y}\rVert \le \frac{\varepsilon}{2}\right)\to 1$  by (\ref{p3eq1}) ($\varepsilon'>0$ is arbitrary). Thus for $n$ large enough (so that $C_RC_F/q\sqrt{g(n^{-1})} \le \frac{\varepsilon^2}{4}$) 

\begin{align*}
\lVert R_{\alpha_n}\bar{Y}_n-K^+\hat{y}\rVert\chi_{\tilde{\Omega}_n} &\le \lVert R_{\alpha_n}\hat{y}-K^+\hat{y}\rVert\chi_{\tilde{\Omega}_n} + \lVert R_{\alpha_n}(\bar{Y}_n-\hat{y})\rVert\chi_{\tilde{\Omega}_n}\\
&\le \frac{\varepsilon}{2}+ \sqrt{\frac{C_RC_F}{\alpha_n}}\delta_n^{true}\chi_{\tilde{\Omega}_n} \le \frac{\varepsilon}{2}+\sqrt{\frac{C_RC_F}{q\sqrt{g(n^{-1})}}} \le \varepsilon,
\end{align*}

\noindent and $\mathbb{P}\left(\lVert R_{\alpha_n}\bar{Y}_n-K^+\hat{y}\rVert \le \varepsilon\right)\ge \mathbb{P}\left(\tilde{\Omega}_n\right)\to 1$ for $n\to\infty$.
\qed
\end{proof}

\subsection{Proofs for the emergency stop case}\label{subsec41}
Again, denote by $\alpha_n$ the output of Algorithm 1 without the emergency stop. For the emergency stop, we have to consider $\lVert R_{\max\{\alpha_n,1/n\}}\bar{Y}_n-K^+\hat{y}\rVert$. It suffices to show that $\mathbb{P}\left(\alpha_n\ge 1/n\right)\to 1$ for $n\to\infty$.\\
 First assume that $K^+\hat{y}=(K^*K)^\frac{\nu}{2}w$ for some $w\in\mathcal{X}$ with $\lVert w \rVert \le \rho$ and $0<\nu\le \nu_0-1$. With (\ref{th4eq2}) it follows that 

\begin{equation}\label{estopeq1}
\mathbb{P}\left( \alpha_n \ge q\left( \frac{\gamma}{4 \rho C_{\nu+1} \sqrt{n}}\right)^\frac{2}{\nu+1}\right)\ge \mathbb{P}\left( \Omega_n\right)\to 1
\end{equation}

\noindent for $n\to\infty$, with $\Omega_n$ given in (\ref{th4omega1}). Otherwise, if there are no such $\nu,
\rho$ and $w$, then (\ref{th3eq1}) implies that for all $\varepsilon>0$ 

\begin{equation}\label{estopeq2}
\mathbb{P}\left( \alpha_n \ge qg(n^{-1})/n\right)\ge\mathbb{P}\left( \Omega_n\right)\to 1
\end{equation}

\noindent for $n\to\infty$, with $g(n^{-1})\to\infty$ and $\Omega_n$ given in (\ref{th3omega1}). Then (\ref{estopeq1}) and (\ref{estopeq2}) together yield $\mathbb{P}\left(\alpha_n\ge 1/n\right)\to 1$ for $n\to\infty$ and therefore the result.
\qed

\subsection{Proof of Corollary \ref{cordisc}}

\begin{proof}[Corollary \ref{cordisc}]
Fix $\varepsilon>0$. Denote by $\alpha_n$ the output of the discrepancy principle with emergency stop and set 
\begin{equation}\label{ceq0}
\Omega_n:=\{ \lVert R_{\alpha_n}\bar{Y}_n-K^+\hat{y}\rVert \le \varepsilon\}.
\end{equation}
\noindent It is
\begin{equation}\label{ceq1}
\lVert R_{\alpha}\hat{y}-K^+\hat{y}\rVert \le \lVert R_{\alpha}K-Id\rVert \lVert \hat{x} \rVert  \le C
\end{equation}

\noindent for all $\alpha>0$. By the triangle inequality,
\begin{align*}
\E \lVert R_{\alpha_n}\bar{Y}_n - K^+\hat{y}\rVert^2 &= 2\E \lVert R_{\alpha_n}\bar{Y}_n - R_{\alpha_n}\hat{y}\rVert^2 + 2\E\lVert R_{\alpha_n}\hat{y}-K^+\hat{y}\rVert^2\\
                                                     &\le  2\E \left[\lVert R_{\alpha_n}\rVert^2{\delta_n^{true}}^2\right] + 2C^2
                                                     \le 2C_RC_F\E\left[ {\delta_n^{true}}^2/\alpha_n\right] + 2C^2\\
                                                      &\le 2nC_RC_F\E {\delta_n^{true}}^2 + 2C^2= 2C_RC_F \E \lVert Y_1-\hat{y}\rVert^2 + 2C^2\le C',
\end{align*}

\noindent where $C'$ does not depend on $n$ and where we used $\alpha_n\le 1$ and (\ref{ceq1}) in the second step and $\alpha_n\ge1/n$ in the fourth. By  (\ref{ceq0}) there holds $\lVert R_{\alpha_n}\bar{Y}_n-K^+\hat{y}\rVert\chi_{\Omega_n} \le \varepsilon$, so

\begin{align*}
\E \lVert R_{\alpha_n}\bar{Y}_n-K^+\hat{y}\rVert^2 & = \E\left[\lVert R_{\alpha_n}\bar{Y}_n-K^+\hat{y}\rVert^2\chi_{\Omega_{n}}\right] + \E\left[\lVert R_{\alpha_n}\bar{Y}_n-K^+\hat{y}\rVert^2\chi_{\Omega_{n}^C}\right]\\
&\le \varepsilon^2 +\E\left[\lVert R_{\alpha_n}\bar{Y}_n-K^+\hat{y}\rVert^2\chi_{\Omega_{n}^C}\right].
\end{align*}
\noindent We apply Cauchy-Schwartz to the second term

\begin{align*}
\E\left[\lVert R_{\alpha_n}\bar{Y}_n-K^+\hat{y}\rVert^2\chi_{\Omega_{n}^C}\right] &\le \sqrt{\E\lVert R_{\alpha_n}\bar{Y}_n-K^+\hat{y}\rVert^4\E\chi_{\Omega_{n}^C}^2}\\
&=\sqrt{\E\lVert R_{\alpha_n}\bar{Y}_n-K^+\hat{y}\rVert^4~\mathbb{P}\left(\Omega_{n}^C\right)}
\end{align*}

\noindent and we claim that there is a constant $A$ with $\E\lVert R_{\alpha_n}\bar{Y}_n-K^+\hat{y}\rVert^4\le A$ for all $n\in\mathbb{N}$.

\begin{align*}
&\E \lVert R_{\alpha_n}\bar{Y}_n-K^+\hat{y}\rVert^4\\
\le & 4\left(\E\lVert R_{\alpha_n}\bar{Y}_n-R_{\alpha_n}\hat{y}\rVert^4+2\E\left[ \lVert R_{\alpha_n}\bar{Y}_n-R_{\alpha_n}\hat{y}\rVert^2\lVert R_{\alpha_n}\hat{y}-K^+\hat{y}\rVert^2\right]+\E \lVert R_{\alpha_n}\hat{y}-K^+\hat{y}\rVert^4\right)\\
\le &4\left( \E \left[\lVert R_{\alpha_n}\rVert^4 {\delta_n^{true}}^4\right] + 2C^2\E\left[ \lVert R_{\alpha_n}\rVert^2 {\delta_n^{true}}^2\right] +C^4\right)\\
\le & B\left( \E\left[ {\delta_n^{true}}^4/\alpha_n^2\right] + \E\left[ {\delta_n^{true}}^2/\alpha_n\right] + 1\right)
\end{align*}
\noindent for some constant $B$, where we used (\ref{ceq1}) in the second step. First,

\begin{align*}
& \E\left[ {\delta_n^{true}}^4/\alpha_n^2\right]\\
 \le & n^2 \E \lVert \bar{Y}_n-\hat{y}\rVert^4 = n^2\E\left[ \sum_{j,j'\ge 1}\left(\bar{Y}_n-\hat{y},u_j\right)^2\left(\bar{Y}_n-\hat{y},u_{j'}\right)^2\right] \\
   = &\frac{1}{n^2}\left(\sum_{j,j'\ge 1} \sum_{i,i',l,l'=1}^n \E\left[\left(Y_i-\hat{y},u_j\right)\left(Y_l-\hat{y},u_j\right)\left(Y_{i'}-\hat{y},u_{j'}\right)\left(Y_{l'}-\hat{y},u_{j'}\right)\right]\right)\\
   \le &\frac{1}{n^2}\sum_{j,j'\ge 1} \left( n \E\left[ \left(Y_1-\hat{y},u_j\right)^2\left(Y_1-\hat{y},u_{j'}\right)^2\right] + n^2 \E\left[\left(Y_1-\hat{y},u_j\right)^2\right]\E\left[\left(Y_1-\hat{y},u_{j'}\right)^2\right]\right.\\
   &\qquad+ \left.2n^2 \left(\E \left[\left(Y_1-\hat{y},u_j\right)\left(Y_1-\hat{y},u_{j'}\right)\right]\right)^2\right)\\
   \le &\frac{n+2n^2}{n^2}\E\left[ \sum_{j,j'\ge1} \left(Y_1-\hat{y},u_j\right)^2\left(Y_1-\hat{y},u_{j'}\right)^2\right]\\
   &\qquad+\E\left[\sum_{j\ge1}\left(Y_1-\hat{y},u_j\right)^2\right]\E\left[\sum_{j'\ge 1}\left(Y_1-\hat{y},u_{j'}\right)^2\right]\\
     \le &\frac{n+2n^2}{n^2}\E\left[\left( \sum_{j\ge1} \left(Y_1-\hat{y},u_j\right)^2\right)^2\right]+\left(\E\left[\sum_{j\ge 1}\left(Y_1-\hat{y},u_j\right)^2\right]\right)^2\\
     = &\frac{n+2n^2}{n^2}\E\lVert Y_1-\hat{y}\rVert^4 + \left(\E\left[ \lVert Y_1-\hat{y}\rVert^2\right]\right)^2 \le B_1
\end{align*}

\noindent for some constant $B_1$, where in the fourth step we used that the $Y_i$ are i.i.d, that $\E\left(Y_1-\hat{y},u_j\right)=\left(\E [Y_1]-\hat{y},u_j\right)=0$ and that $\E[XY]=\E[X]\E[Y]$ for independent (and integrable) random variables 
(so the relevant cases are the ones where either all indices $i,i',l,l'$ are equal or exactly pairwise two).
Then we used Jensen's inequality in the fifth step. Moreover,\newline $\E\left[ {\delta_n^{true}}^2/\alpha_n\right]\le n \E\left[{\delta_n^{true}}^2\right]= \E \lVert Y_1-\hat{y}\rVert^2=B_2$, so the claim holds for $A=B(B_1+B_2+1)$. 
By Theorem 3 it holds that $\mathbb{P}\left(\Omega_n\right)\to 1$ for $n\to\infty$, thus  $\mathbb{P}\left(\Omega_n^C\right)\le \varepsilon^4/A$ for $n$ large enough and
\begin{align*}
\E \lVert R_{\alpha_n}\bar{Y}_n-K^+\hat{y}\rVert^2 &\le \varepsilon^2 \E[\chi_{\Omega_n}]+\sqrt{\E\lVert R_{\alpha_n}\bar{Y}_n-K^+\hat{y}\rVert^4~\mathbb{P}\left(\Omega_{n}^C\right)} \le 2\varepsilon^2.
\end{align*}
\noindent \qed
\end{proof}

\section{Numerical demonstration}
We conclude with some numerical results.

\subsection{Differentiation of binary option prices}

A natural example is given if the data is acquired by a Monte-Carlo simulation, here we consider an example from mathematical finance. The buyer of a binary call option receives after $T$ days a payoff $Q$, if then a certain stock price $S_T$ is higher then the strike value $K$. Otherwise he gets nothing. Thus the value $V$ of the binary option depends on the expected evolution of the stock price. We denote by $r$ the riskfree rate, for which we could have invested the buying price of the option until the expiry rate $T$. If we already knew today for sure, that the stock price will hit the strike (insider information), we would pay $V=e^{-rT}Q$ for the binary option ($e^{-rT}$ is called discount factor). Otherwise, if we believed that the stock price will hit the strike with probability $p$, we would pay $V=e^{-rT}Qp$. In the Black Scholes model one assumes, that the relative change of the stock price in a short time intervall is normally distributed, that is

\begin{equation*}
S_{t+\delta t}-S_t \sim \mathcal{N}(\mu \delta t,\sigma^2 \delta t).
\end{equation*}

\noindent Under this assumption one can show that (see \cite{hull2016options})

\begin{equation*}
S_T = S_0 e^{sT},
\end{equation*}

\noindent where $S_0$ is the initial stock price and $s \sim\mathcal{N}\left(\mu-\sigma^2/2,\sigma^2/T\right)$. Under this assumptions one has $V=e^{-rT}Q\Phi(d)$, with

\begin{equation*}
\Phi(x):=\frac{1}{\sqrt{2\pi}}\int_{-\infty}^x e^{-\frac{\xi^2}{2}}d\xi,\qquad d=\frac{\log \frac{S_0}{K}+T\left(\mu-\frac{\sigma^2}{2}\right)}{\sigma\sqrt{T}}.
\end{equation*}

\noindent Ultimatively we are interested in the sensitivity of $V$ with respect to the starting stock price $S_0$, that is $\partial V(S_0)/\partial S_0$. We formulate this as the inverse problem of differentiation. Set $\mathcal{X}=\mathcal{Y}=L^2([0,1]=$ and define 

\begin{align*}
 K: &L^2([0,1])\to L^2([0,1])\\
    &f \mapsto Af=g: x \mapsto \int_0^xf(y)dy.
\end{align*}

\noindent Then our true data is $\hat{y}=V=e^{-rT}Q\Phi(d)$. To demonstrate our results we now approximate $V: S_0\mapsto e^{-rT}Qp(S_0)$ through a Monte-Carlo approach. That is we generate independent gaussian random variables $Z_1,Z_2,...$ identically distributed to $s$ and set $Y_i:=e^{-rT}Q \chi_{\{S_0e^{TZ_i}\ge K \}}$.
Then we have $\E Y_i = e^{-rT}Q\mathbb{P}(S_0e^{TZ_i})=e^{-rT}Qp(S_0)=V(S_0)$ and $\E\| Y_i \|^2\le e^{-rT}Q<\infty$. We replace $L^2([0,1])$ with piecewise continuous linear 
splines on a homogeneous grid with $m=50000$ elements (we can calculate $Kg$ exactly for such a spline $g$).  We use in total $n=10000$ random variables for each simulation. As parameters we chose $r=0.0001, T=30, K=0.5,
Q=1, \mu= 0.01, \sigma=0.1$. It is easy to see that $\hat{x} =K^+\hat{y}\in \mathcal{X}_{\nu}$
for all $\nu >0$ using the transformation $z(\xi)=0,5e^{\sqrt{0,3}\xi-0,15}$. Since the qualification of the Tikhonov regularisation is $2$, Theorem \ref{thws} gives an error bound which is asymptotically
proportional to $\left(1/\sqrt{n}\right)^\frac{1}{2}$. In Figure 1 we plot the $L^2$ average of 100 simulations of the discrepancy principle together with the (translated) optimal error bound. In this case the emergency stop did not trigger once - this is plausible, since the true solution is very smooth, which yields comparably higher values of the regularisation parameter and also, the error distribution is Gaussian and the problem is only mildly ill-posed.\\
Let us stress that this is only an academic example to demonstrate the possibility of using our new generic approach in the context of Monte Carlo simulations. Explicit solution formulas for standard binary options are well-known, and for more complex financial derivatives with discontinuous payoff profiles (such as autocallables or Coco-bonds) one would rather resort to stably differentiable Monte Carlo methods (\cite{alm2013monte} or \cite{gerstner2018monte}) or use specific regularization methods for numerical differentiation \cite{hanke2001inverse}.

\begin{figure}\label{risk}
 \includegraphics[width=1\textwidth]{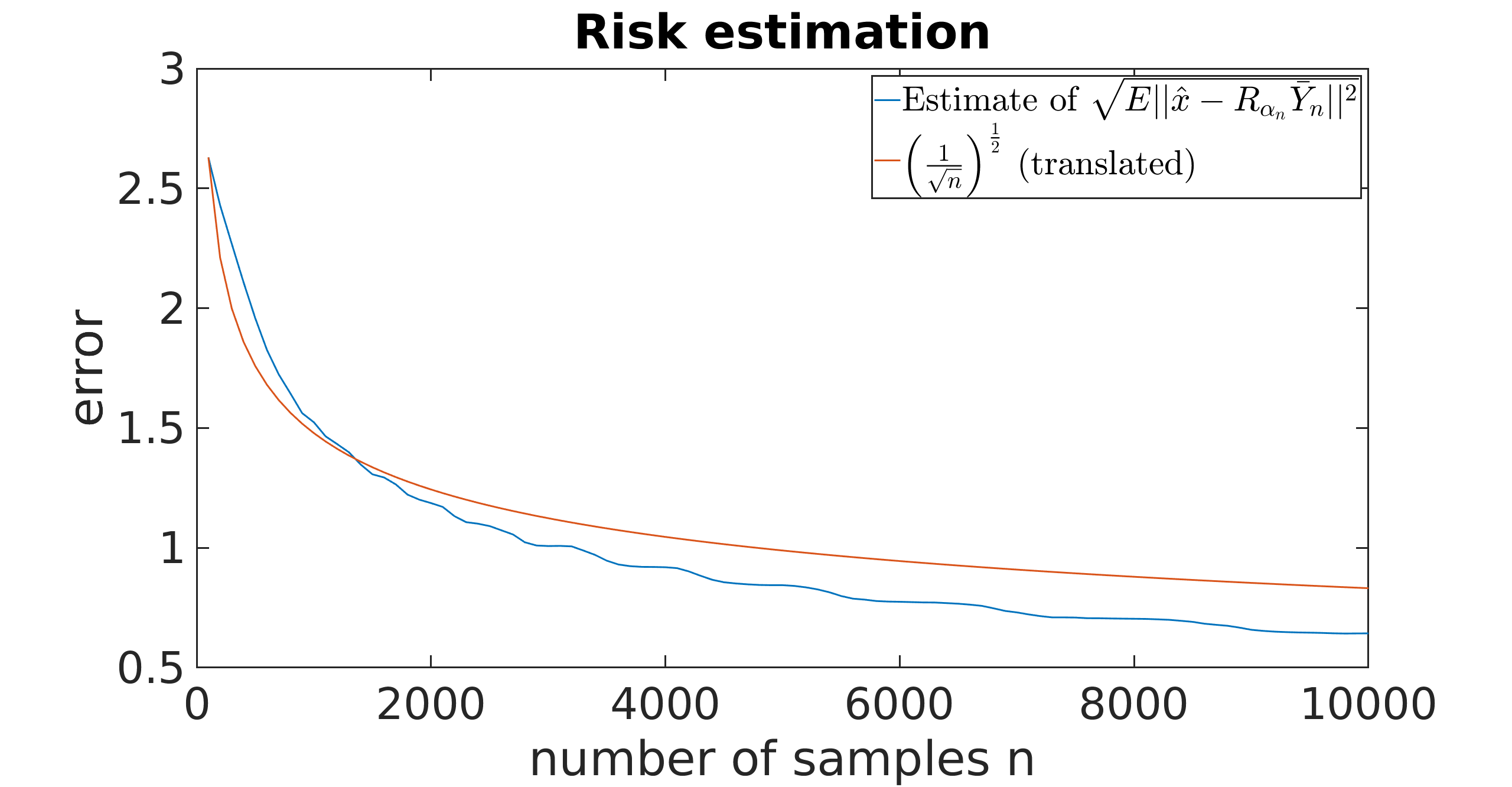}
 \caption{ Estimated Risk of a binary option.}
\end{figure}

\subsection{Inverse heat equation}
We consider the toy problem 'heat' from \cite{hansen2010discrete}. We chose the discretisation level $m=100$ and set $\sigma=0.7$. Under this choice, the last seven singular values (calculated with the function 'csvd') fall below the machine precision of $10^{-16}$.
The discretised large systems of linear equations are solved iteratively using the conjugate gradient method ('pcg' from MATLAB) with a tolerance of $10^{-8}$. As a regularisation method we chose Tikhonov regularisation and we compared the a priori choice $\alpha_n=1/\sqrt{n}$, the discrepancy principle (dp) and the discrepancy principle with emergency stop (dp+es), as implemented in Algorithm 1 with $q=0.7$ and estimated sample variance.
The unbiased i.i.d measurements fullfill $\sqrt{\E \lVert Y_i-\hat{y}\rVert^2}\approx 1.16$ and $\E \lVert Y_i - \E Y_i \rVert^k=\infty$ for $k\ge 3$. Concretely, we chose $Y_i:=\hat{y}+E_i$ with $E_i:=U_i*Z_i*v$, where the $U_i$ are independent and uniformly on $[-1/2,1/2]$ distributed, the $Z_i$ are independent Pareto distributed (MATLAB function 'gprnd' with parameters 1/3, 1/2 and 3/2), and $v$ is a uniform permutation of $1,1/2^\frac{3}{4},...,1/m^\frac{3}{4}$. Thus we chose a rather ill-posed problem together with a heavy-tailed error distribution.
We considered three different sample sizes $n=10^3,10^4,10^5$ with $200$ simulations for each one. The results are presented as boxplots in Figure 3. It is visible, that the results are much more concentrated for a priori regularisation and discrepancy prinicple with emergency stop, indicating the $L^2$ convergence (strictly speaking we do not know if the discrepancy principle with emergency stop converges in $L^2$, since the additional assumption of Corollary \ref{cordisc} is violated here). Moreover the statistics of the discrepancy principle with and without emergency stop become more similiar with increasing sample size - with the crucial difference, that the outliers as such we denote the red crosses above the blue box, thus the cases where the mehod performed badly) are only present in case of the discrepancy principle without emergency stop, causing non-convergence in $L^2$, see Figure 2. Thus here the discrepancy principle with emergency stop is superior to the discrepancy principle without emergency stop, in particular for large sample sizes. Beside that, the error is falling slower in case of the a priori parameter choice.
The number of outliers falls with increasing sample size from $37$ for $n=10^3$  to $18$ for $n=10^5$, indicating the (slow) convergence in probability of the discrepancy principle. Note that $\delta_n^{true}/\delta_n^{est}\approx 1.9$ (in average), if we only consider the runs yielding outliers. This illustrates, that the lack of convergence in $L^2$ is caused by the occasional underestimation of the data error.

\begin{figure}
\caption{Estimated relative $L^2$ error for 'heat', that is $\sqrt{\sum_{t=1}^{200} e_i^2/200}$ where $e_i$ is the relative error $\lVert R_{\alpha_n}\bar{Y}_n-K^+\hat{y}\rVert\rVert/\lVert K^+\hat{y}\rVert$ of the $i$-th run.}
		\begin{tabular}{lrrr}
			         & dp & dp+es & a priori \\
			$n=10^3$    & 572.49     & 0.66 & 0.83\\
			$n=10^4$  & 79.45      & 0.49 & 0.76\\
			$n=10^5$     & 107.19     & 0.31 & 0.69\\
		\end{tabular}
\end{figure}

\begin{figure}
\label{heat1}
	\begin{minipage}{0.49\linewidth}
		\centering
		\includegraphics[width=\linewidth]{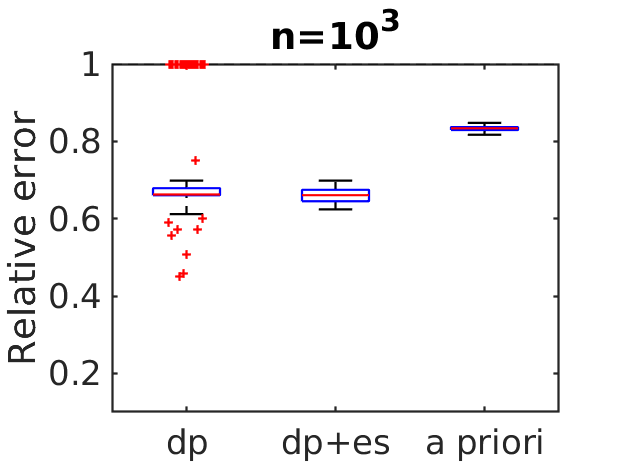}
	\end{minipage}
        \begin{minipage}{0.49\linewidth}
        \includegraphics[width=\linewidth]{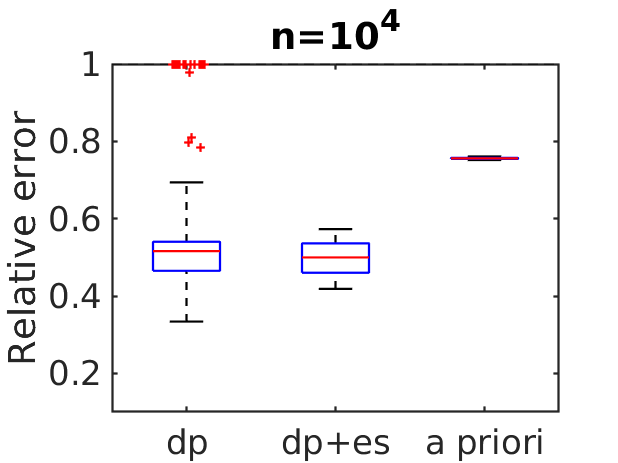}
         \end{minipage}
	\begin{minipage}{0.49\linewidth}
		\includegraphics[width=\linewidth]{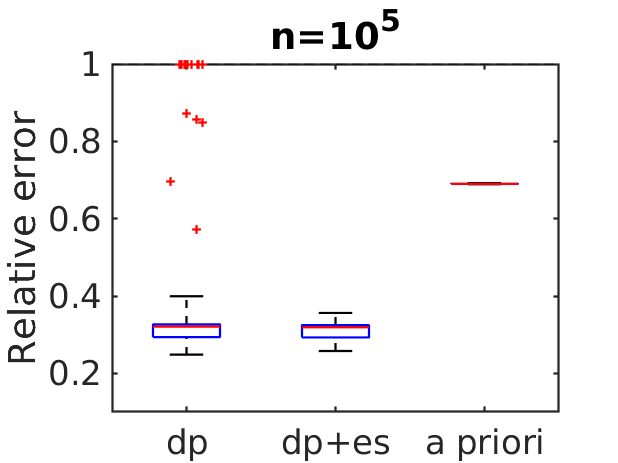}
	\end{minipage}
	\caption{Comparison of Tikhonov regularisation with discrepancy principle (dp, Algorithm 1), discrepancy principle with emergency stop (dp+es, Algorithm 1 (optional)) and a priori choice for 'heat'. Boxplots of the relative errors $\lVert R_{\alpha_n}\bar{Y}_n-K^+\hat{y}\rVert/\lVert K^+\hat{y}\rVert$ for 200 simulations with three different sample sizes.}
\end{figure}
\FloatBarrier
\bibliographystyle{spmpsci}
\bibliography{references}

\end{document}